\documentclass[11pt,a4paper]{article}

\usepackage[left=35mm,right=35mm,top=30mm,bottom=30mm]{geometry}
\usepackage{epsfig}
\usepackage{epstopdf}

\usepackage{xfrac}
\usepackage{float}

\usepackage{color}
\usepackage{amsthm,amsmath,amssymb}
\usepackage{mathrsfs}
\usepackage{sectsty}
\usepackage{hyperref}

\usepackage[hang,flushmargin]{footmisc}
\usepackage{enumitem}
\usepackage{titlesec}
	\titlespacing{\section}{0pt}{12pt}{0pt}
	\titlespacing{\subsection}{0pt}{6pt}{0pt}
	
\titlelabel{\thetitle.\quad}

\makeatletter

\long\def\@footnotetext#1{%
\H@@footnotetext{%
\ifHy@nesting
\hyper@@anchor{\@currentHref}{#1}%
\else
\Hy@raisedlink{\hyper@@anchor{\@currentHref}{\relax}}#1%
\fi
}}

\def\@footnotemark{%
\leavevmode
\ifhmode\edef\@x@sf{\the\spacefactor}\nobreak\fi
\H@refstepcounter{Hfootnote}%
\hyper@makecurrent{Hfootnote}%
\hyper@linkstart{link}{\@currentHref}%
\@makefnmark
\hyper@linkend
\ifhmode\spacefactor\@x@sf\fi
\relax
}%

\ifFN@multiplefootnote%
\renewcommand*\@footnotemark{%
\leavevmode
\ifhmode
\edef\@x@sf{\the\spacefactor}%
\FN@mf@check
\nobreak
\fi
\H@refstepcounter{Hfootnote}%
\hyper@makecurrent{Hfootnote}%
\hyper@linkstart{link}{\@currentHref}%
\@makefnmark
\hyper@linkend
\ifFN@pp@towrite
\FN@pp@writetemp
\FN@pp@towritefalse
\fi
\FN@mf@prepare
\ifhmode\spacefactor\@x@sf\fi
\relax%
}%
\fi

\makeatother

\theoremstyle{plain}
\newtheorem{theorem}{Theorem}[section]

\newtheorem{lemma}[theorem]{Lemma}

\theoremstyle{definition}

\newtheorem{remark}[theorem]{Remark}

\sectionfont{\large \sc}
\subsectionfont{\normalsize}

\long\def\symbolfootnote[#1]#2{\begingroup%
\def\thefootnote{\fnsymbol{footnote}}\footnote[#1]{#2}\endgroup}

\def\blfootnote{\xdef\@thefnmark{}\@footnotetext}

\begin{document}

{\Large \bfseries \sc Growth rates of Coxeter groups and Perron numbers\\}

{\bfseries Alexander Kolpakov and Alexey Talambutsa\\}

{\em Abstract.}
We define a large class of abstract Coxeter groups, that we call $\infty$--spanned, and for which the word growth rate and the geodesic growth rate appear to be Perron numbers. This class contains a fair amount of Coxeter groups acting on hyperbolic spaces, thus corroborating a conjecture by Kellerhals and Perren. We also show that for this class the geodesic growth rate strictly dominates the word growth rate.
\vspace{0.25in}

\section{Introduction}\label{sec:intro}

A Coxeter group $G$ of rank $n$ is an abstract group that can be defined by the generators $S = \{ s_1, s_2, \dots, s_n \}$ and relations as follows:
\begin{equation}
G = \langle  s_1, s_2, \dots, s_n\, |\, s^2_i = 1, (s_i s_j)^{m_{ij}} = 1,\, 1 \leq i < j \leq n \rangle,
\end{equation}
for all $1 \leq i \leq n$, and $m_{ij} \in \{2, 3, \dots \} \cup \{ \infty \}$, for all $1 \leq i < j \leq n$. No relation is present between $s_i$ and $s_j$ if and only if $m_{ij} = \infty$.

Such a group can be conveniently described by its Coxeter diagram $\mathcal{D}$, which is a labelled graph, where each vertex $i$ corresponds to a generator $s_i$ of $G$, with $i$ and $j$ connected by an edge whenever $m_{ij} \geq 3$. Moreover, if $m_{ij} \geq 4$ then the edge joining $i$ and $j$ has label $m_{ij}$, while for $m_{ij} = 3$ it remains unlabelled.

We say that $G$ is \textit{$\infty$--spanned}, if the Coxeter diagram for $G$ has a spanning tree with edges labelled only $\infty$ and contains more than $2$ vertices. The corresponding Coxeter diagram will be also called $\infty$--spanned. In particular, neither the finite group $C_2$ with two elements, nor the infinite dihedral group $D_\infty$ are considered $\infty$--spanned.  Any other infinite right-angled Coxeter group that cannot be decomposed as a direct product is an $\infty$--spanned group.

Given a Coxeter group of rank $n$ with generating set $S = \{ s_1, s_2, \dots, s_n \}$ of involutions, called its \textit{standard generating set}, let us consider its Cayley graph $\mathrm{Cay}(G, S)$ with the identity element $e$ as origin and the word metric $d(g,h) = $ ``the least length of a word in the alphabet $S$ necessary to write down $gh^{-1}$''. Let the \textit{word length} of an element $g \in G$ be $d(e, g)$. Then, let $w_k$ denote the number of elements in $G$ of word length $k \geq 0$ (assuming that $w_0 = 1$, so that the only element of zero word length is $e$). Also, let $g_k$ denote the number of geodesic paths in the graph $\mathrm{Cay}(G, S)$ of length $k \geq 0$ issuing from $e$ (with a unique geodesic
of length $0$ being the point $e$ itself, and thus $g_0 = 1$).

The \textit{word growth series} of $G$ with respect to its standard generating set $S$ is
\begin{equation}
\omega_{(G, S)}(z) = \sum^\infty_{k=0} w_k z^k,
\end{equation}
while the \textit{geodesic growth series} of $G$ with respect to $S$ is
\begin{equation}
\gamma_{(G, S)}(z) = \sum^\infty_{k=0} g_k z^k.
\end{equation}

Since we shall always use a fixed standard generating $S$ set for $G$ in the sequel, and mostly refer to the Coxeter diagram defining $G$, rather than $G$ itself, we simply write $\omega_{G}(z)$ and $\gamma_{G}(z)$ for its word and geodesic generating series. As well, by saying that $G$ is $\infty$--spanned we shall refer to its Coxeter diagram.

The limiting value $\omega(G) = \limsup_{k\to \infty} \sqrt[k]{w_k}$ is called the word growth rate of $G$, while $\gamma(G) = \limsup_{k\to \infty} \sqrt[k]{g_k}$ is called the geodesic growth rate of $G$. Note that these two upper limits can be replaced by the usual limits according to Fekete's lemma, which guarantees that the limit $\lim_{k\to \infty}\sqrt[k]{a_k}$ exists for a non-negative sequence $a_k$ that is submultiplicative, i.e. such that $a_{m+n}\leq a_m a_n$ for all $m, n\geq 0$.

Both growth series above are known to be rational functions, since the corresponding sets $\mathrm{ShortLex}(G) = $ ``words over the alphabet $S$ in shortest left-lexicographic form representing all elements of $G$'' (equivalently, the language of shortlex normal forms for $G$ with its standard presentation) and $\mathrm{Geo}(G) = $``words over the alphabet $S$ corresponding to labels of all possible geodesics in $\mathrm{Cay}(G, S)$ issuing from $e$'' (equivalently, the language of reduced words in $G$ with its standard presentation) are regular languages. That is, there exist deterministic finite-state automata $\mathsf{ShortLex}$ and $\mathsf{Geo}$ that accept the eponymous languages. We shall use the automata due to Brink and Howlett \cite{BH},  who provide a universal construction for all Coxeter groups. In addition to the original work \cite{BH}, the $\mathsf{ShortLex}$ automaton is described in great detail in \cite{Casselman1, Casselman2}, and the $\mathsf{Geo}$ automaton in \cite{BB}. In the present work, we use those latter descriptions.

Given a finite automaton $A$ over an alphabet $S$, let $L = L(A)$ be its accepted language. If $v_k$ is the number of length $k \geq 0$ words over $S$ that belong to $L$, then the quantity $\lambda(A) = \limsup_{k\to \infty} \sqrt[k]{v_k}$ is called the growth rate of the (regular) language $L(A)$. From the above discussion we have $\omega(G) = \lambda(\mathrm{ShortLex})$ and $\gamma(G) = \lambda(\mathrm{Geo})$.

Growth rates of many classes of Coxeter groups are known to belong to classical families of algebraic integers, in particular, to \textit{Perron numbers}. Moreover, growth rates of Coxeter groups acting cocompactly on hyperbolic space $\mathbb{H}^d$, for $d \geq 4$, are specifically conjectured to be Perron numbers by Kellerhals and Perren \cite{KePe}. Here, following Lind's work \cite{L}, we define \textit{a Perron number} to be a real algebraic integer $\tau \geq 1$ with all its other Galois conjugates being strictly less than $\tau$ in absolute value. Perron numbers often appear in the context of harmonic analysis \cite{Bertin}, dynamical systems \cite{LM}, arithmetic groups \cite{ERT}, and many others.

It follows from the results of \cite{Floyd, Parry, Yu1, Yu2} that the growth rates of Coxeter groups acting on $\mathbb{H}^2$ and $\mathbb{H}^3$ with finite co-volume are Perron numbers. The main goal of this paper is to prove the following theorem, that partially confirms the aforementioned conjecture, and also extends to the case of geodesic growth rates.

\begin{theorem}\label{thm:Perron}
Let $G$ be an $\infty$--spanned Coxeter group. Then $\omega(G)$ and $\gamma(G)$ are Perron numbers such that $\omega(G)\geq (1+\sqrt5)/2$ and $\gamma(G)\geq \gamma_0$, where $\gamma_0\approx 1.769$ is the real root of the polynomial $x^3-2x-2$.
\end{theorem}

The previous works \cite{BrKe, Floyd, KePe, Parry, Yu1, Yu2} make use of Steinberg's formula \cite{Steinberg} as their main tool to investigate the combinatorial composition of the growth function $\omega_G(z)$, also invoking the classification of vertex stabilisers for the corresponding fundamental domains in $\mathbb{H}^n$,  $n = 2, 3$. Such classification becomes much more complicated in higher dimensions and the number of terms in Steinberg's formula grows fast for higher rank Coxeter groups, which makes this approach hardly tractable.

Our method is different and does not rely on the geometric action of the group. Instead, it makes essential use of Perron--Frobenius theory applied to the adjacency matrices of the finite automata \textsf{Geo} and \textsf{ShortLex}. Below we recall the setting of Perron--Frobenius theory and remind the classical result which we use.

A non-negative matrix $M$ is called \textit{reducible} if there exists a permutation matrix $P$ such that $PMP^{-1}$ has an upper-triangular block form. Otherwise, $M$ is called \textit{irreducible (or indecomposable)}. If $M$ is the adjacency matrix of a directed graph $D$, then $M$ is irreducible if and only if $D$ is strongly connected. The \textit{$i$--th period} ($1 \leq i \leq n$) of a non-negative matrix $M$ is the greatest common divisor of all natural numbers $d$ such that $(M^d)_{ii} > 0$. If $M$ is irreducible, then the periods of $M$ are all equal to \textit{the period of} $M$. A non-negative matrix is called \textit{aperiodic} if it has period $1$.
A non-negative matrix that is irreducible and aperiodic is called \emph{primitive}.
The classical Perron-Frobenius theorem implies that the largest real eigenvalue of a square $n\times n$ ($n \geq 2$) non-negative primitive integral matrix is a Perron number, cf. \cite[Theorem~4.5.11]{LM}.

The combinatorial condition on the Coxeter diagram being $\infty$--spanned is easy to verify even for rather large diagrams. This allows us to apply our results to reflection groups of higher rank, such as the Kaplinskaya--Vinberg example in $\mathbb{H}^{19}$, cf. Section~\ref{sec:geom}. This work expands on the results of \cite{KoTa} and promotes them to much greater generality. To our best knowledge, however, there is no Coxeter group known with growth rate that is not a Perron number (provided it exceeds $1$).

Another question that comes about naturally  is the number $\gamma_G(g)$ of geodesics in $\mathrm{Cay}(G, S)$ issuing from the neutral element $e$ of $G$ and arriving to a given element $g \in G$. It is clear that $\gamma_G(g)$ heavily depends on $g\in G$: if a right-angled Coxeter group $G$ contains the direct product $D_\infty \times D_\infty$ as a parabolic subgroup, then some elements will have a unique geodesic representative, while some will have exponentially many depending on their word length. Nevertheless, the average number of geodesics that represent an element of word length $k$, i.e. the ratio $\frac{g_k}{w_k}$, can be analysed.

\begin{theorem}\label{thm:geodesic}
Let $G$ be an $\infty$--spanned Coxeter group which is not a free product $C_2 * \ldots * C_2$. Then $\frac{g_k}{w_k} \sim \nu \cdot \delta^k(G)$ asymptotically\footnote{Here by writing $a_k \sim b_k$ for two  sequences of positive real numbers indexed by integers, we mean $\lim_{k\to \infty}\frac{a_k}{b_k} = 1$.}, as $k\rightarrow \infty$, with $\nu = \nu(G) > 0$ a positive constant, and $\delta(G) = \frac{\gamma(G)}{\omega(G)} > 1$. In particular, $\gamma(G)$ always strictly dominates $\omega(G)$.
\end{theorem}

The paper is organised as follows: in Section \ref{automata:construction} we describe the deterministic finite-state automata accepting the languages $\mathsf{ShortLex}$ and $\mathsf{Geo}$ (their construction is first given in \cite{BH} the paper by Brink and Howlett), and show some of their properties, essential for the subsequent proofs, in Section \ref{automata:properties}. Then, in Section \ref{proofs}, we prove Theorems \ref{thm:Perron} and \ref{thm:geodesic}. Finally, a few geometric applications are given in Section~\ref{sec:geom}.

\begin{center}
\textsc{Acknowledgements}\\
\end{center}
\noindent
{\small A.K. was partially supported by the Swiss National Science Foundation (project no.~PP00P2-170560) and the Russian Federation Government (grant no. 075-15-2019-1926). A.T. was partially supported by the Russian Foundation for Basic Research, projects no.~18-01-00822 and no.~18-51-05006. The authors would like to thank Alexander A. Gaifullin, Ruth Kellerhals and Tatiana Smirnova-Nagnibeda for stimulating discussions. Also, A.T. would like to thank the University of Neuch\^{a}tel for hospitality during his visit in January 2020. The authors are grateful to the anonymous referees for their remarks and suggestions that helped improving this paper.}

\section{Brink and Howlett's automata and their properties}

In this section we briefly recall the general construction of the automata $\mathsf{ShortLex}$ and $\mathsf{Geo}$ that accept, respectively, the shortlex and geodesic languages for an arbitrary Coxeter group $G$ with standard generating set $S = \{ s_1, s_2, \dots, s_n \}$. Then we shall concentrate on some combinatorial and dynamical properties of those automata in the case when $G$ is $\infty$--spanned. For the background on automata and regular languages and their usage for Coxeter groups we refer the reader to the book \cite{BB}.

\subsection{Constructing the automata}\label{automata:construction}

Let $G$ be a Coxeter group with generators $S = \{ s_1, s_2, \dots, s_n \}$ and presentation
\begin{equation}
G = \langle  s_1, s_2, \dots, s_n\, |\, (s_i s_j)^{m_{ij}} = 1, \mbox{ for } 1 \leq i, j \leq n  \rangle,
\end{equation}
where we assume that $m_{ii} = 1$, for all $1 \leq i \leq n$, and $m_{ij} = m_{ji} \in \{2, 3, \dots \} \cup \{ \infty \}$, for all $1 \leq i < j \leq n$.

Let $V = \mathbb{R}^n$, and let $\{ \alpha_1, \dots, \alpha_n \}$ be a basis in $V$, called the set of \textit{simple roots} of $G$. The associated symmetric bilinear form $(u \,|\, v)$ on $V\times V$ is defined by
\begin{equation}
(\alpha_i \,|\, \alpha_j) = - \cos \frac{\pi}{m_{ij}}, \mbox{ for all } 1 \leq i, j \leq n.
\end{equation}
Let, for each $s_i \in S$, the corresponding \textit{simple reflection} in the hyperplane $H_i$ orthogonal to the root $\alpha_i$ be defined as
\begin{equation}
\sigma_i(v) = v - 2 (v \,|\, \alpha_i) \alpha_i, \mbox{ for } 1 \leq i \leq n.
\label{eq:reflection-formula}
\end{equation}
Then the representation $\rho: G \rightarrow GL(V)$ given by
\begin{equation}
\rho(s_i) = \sigma_i, \mbox{ for } 1 \leq i \leq n,
\end{equation}
is a faithful linear representation of $G$ into the group of linear transformations of $V$, called \textit{the geometric representation}, cf. \cite[\S 4.2]{BB}.

Let us define the set $\Sigma$ of \textit{small roots}\footnote{Small roots are called \textit{minimal roots} in \cite{Casselman1, Casselman2} due to their minimality with respect to the dominance relation introduced in the original paper \cite{BH}.} of $G$ as the minimal (by inclusion) subset of vectors in $V$ satisfying the following conditions:
\begin{itemize}
\item $\alpha_i \in \Sigma$, for each $1 \leq i \leq n$;
\item if $v \in \Sigma$, then $\sigma_i(v) \in \Sigma$, \, for all $1\leq i \leq n$ such that $-1 < (v|\alpha_i) < 0$.
\end{itemize}
In other words, all simple roots of $G$ are small, and if $v$ is a small root of $G$, then $u = \sigma_i(v)$ is also a small root provided that the $i$--th coordinate of $u$ is strictly bigger than the $i$--th coordinate of $v$, and the (positive) difference is strictly less than $2$. Note that each $v \in \Sigma$ is a non-trivial linear combination $\sum^n_{i=1} c_i\,\alpha_i$ with non-negative coefficients $c_i \geq 0$.

The set $\Sigma$ of small roots is known to be finite \cite[Theorem~4.7.3]{BB}. In particular, if $\alpha_i$ and $\alpha_j$ ($i \neq j$) are such two roots that $m_{ij} = \infty$, then $\sigma_i(\alpha_j)$ is \textit{not} a small root. Thus, if $G$ is $\infty$--spanned, we would expect it to have ``not too many'' small roots, so that a more precise combinatorial analysis of the latter becomes possible.

The set of $\mathrm{ShortLex}$ words, as well as the set $\mathrm{Geo}$ of geodesic words, in $G$ are regular languages by \cite[Theorem~4.8.3]{BB}. Each is accepted by the corresponding finite automaton that we shall call, with slight ambiguity, $\mathsf{ShortLex}$ and $\mathsf{Geo}$, respectively. Their states (besides a single state $\star$) are subsets of $\Sigma$ and their transition functions can be described in terms of the action of generating reflections $\sigma_i$, as follows.

For $\mathsf{Geo}$, the start state is $\{ \emptyset \}$, the fail state is $\star$, and the transition function $\delta(D, s_i)$, for a state $D$ and a generator $s_i$, $i=1,\dots,n$, is defined by
\begin{itemize}
\item $\delta(D, s_i) = \star$, if $\alpha_i \in D$ or $D = \star$, or otherwise
\item $\delta(D, s_i) = \{ \alpha_i \} \cup (\{ \sigma_i(v), \, v \in D \} \cap \Sigma)$.
\end{itemize}
All states of $\mathsf{Geo}$, except for $\star$, are accept states. The entire set of states can be obtained by applying the transition function inductively to the start state and its subsequent images. Then the fact that $\Sigma$ is finite \cite[Theorem~4.7.3]{BB} guarantees that the resulting set of states is finite.

For $\mathsf{ShortLex}$, the start state is $\{ \emptyset \}$, the fail state is $\star$, and the transition function $\delta(D, s_i)$, for a state $D$ and a generator $s_i$, $i=1,\dots,n$, is given by
\begin{itemize}
\item  $\delta(D, s_i) = \star$, if $\alpha_i \in D$ or $D = \star$, or otherwise
\item $\delta(D, s_i) = \{ \alpha_i \} \cup \left( \{ \sigma_i(v), \, v \in D \} \cup \{ \sigma_i(\alpha_j), \, j<i \}  \right) \cap \Sigma$.
\end{itemize}
All states of $\mathsf{ShortLex}$,
except for $\star$, are accept states. Again, all other states of $\mathsf{ShortLex}$ can be obtained inductively from the start state.

The enhanced transition function of $\mathsf{ShortLex}$ or $\mathsf{Geo}$ automaton from a state $D$ upon reading a length $l \geq 1$ word $w$ over the alphabet $S$ will be denoted by $\widehat{\delta}(D, w)$. It is inductively defined by first setting $\widehat{\delta}(D, s_i) = \delta(D, s_i)$, for all $i=1, \dots, n$ and, if $l\geq 2$, by $\widehat{\delta}(D, w) = \delta(\widehat{\delta}(D, w'), s_i)$, where $w = w' s_i$ for a word $w'$ of length $l-1$ and a generator $s_i$ with some appropriate $i \in \{1, 2, \dots, n\}$.

We  refer the reader to the original work \cite{BH}, and also the subsequent works \cite{Casselman1, Casselman2} for more detail on the above constructions. A very informative description of geodesic automata can be found in \cite[\S 4.7--4.8]{BB}.

For the sake of convenience, we shall omit the fail state $\star$ and the corresponding transitions in all our automata. This will make many computations in the sequel simpler, since we care only about the number of accepted words.

\subsection{Auxiliary lemmas}\label{automata:properties}

If $\Gamma$ is a tree, i.e. a connected graph without closed paths of edges, a vertex of $\Gamma$ having degree $1$ is called a leaf of $\Gamma$. The set of leaves of $\Gamma$, which is denoted by $\partial \Gamma$, is called the boundary of $\Gamma$.

\begin{lemma}[Labelling lemma]\label{lemma:labelling}
Let $\mathcal{D}$ be an $\infty$--spanned diagram with vertices $\{1, 2, \dots, n\}$, with $n\geq 3$, and $\Gamma \subset \mathcal{D}$ be its spanning tree all of whose edges have labels $\infty$. Then, up to a renumbering of vertices, we may assume that $\Gamma$ contains the edges $1 \rightarrow 2$ and $2 \rightarrow 3$, and for any non-recurring path $i_0 = 1 \rightarrow i_1 \rightarrow i_2 \rightarrow \dots \rightarrow i_k$ inside $\Gamma$, such that $i_k \in \partial \Gamma$, we have $i_0 < i_1 < i_2 < \dots < i_k$.
\end{lemma}
\begin{proof}
We explicitly construct the desired enumeration. Choose two edges forming a connected sub-tree of $\Gamma$ and label their vertices $1$, $2$ and $3$, such that vertex $2$ is between the vertices $1$ and $3$. Then start labelling the leaves in $\partial \Gamma$ by assigning numbers to them down from $n$. When all the leaves are labelled, form a new tree $\Gamma^\prime = \Gamma - \partial \Gamma$, and label the leaves in $\partial \Gamma^\prime$, and so on, until no unused labels remain.
\end{proof}

From now on, we shall suppose that every $\infty$--spanned diagram with $3$ or more vertices already has a labelling satisfying Lemma~\ref{lemma:labelling}. Such a labelling will become handy later on. By $\Gamma$ we will be denoting the corresponding spanning tree.

\begin{lemma}[Hiking lemma]\label{lemma:hiking}
Let $D' = \delta(D, s_i)$ be an accept state of the automaton $\mathsf{ShortLex} = \mathsf{ShortLex}(\mathcal{D})$, resp.  $\mathsf{Geo} = \mathsf{Geo}(\mathcal{D})$. Then for any vertex $j$ that is adjacent to $i$ in the tree $\Gamma$, the state $D'' = \delta(D', s_j) \neq D'$ is also an accept state of $\mathsf{ShortLex}$, resp. $\mathsf{Geo}$.
\end{lemma}
\begin{proof}
By definition, all states of $\mathsf{ShortLex}$ and $\mathsf{Geo}$, except for the fail states $\star$, are accepting. If $D = \{ \emptyset \}$ is the start state, there is no sequence of transition bringing the automaton back to it, by definition. Now we need to check that $s_j\notin D'$, which shows that $D''\ne \star$. Indeed, supposing the contrary, we would have $\sigma_i(\alpha) = \alpha_j$ or, equivalently $\alpha = \sigma_i(\alpha_j) = \alpha_j + 2 \alpha_i$, for a small root $\alpha \in D$. The latter is impossible since $(\alpha_j + 2 \alpha_i \,|\, \alpha_i)=1$, which contradicts the inequality $(\alpha|\alpha_i)<1$ that holds true for any short root $\alpha\ne \alpha_i$ (see \cite[Lemma 4.7.1]{BB}). Since $s_j\notin D'$, and $s_j\in \delta(D', s_j)=D''$, we also obtain that $D''\ne D'$.
\end{proof}

The main upshot of Lemma~\ref{lemma:hiking} is that we can repeatedly apply the generators which are connected in $\Gamma$, and thus move between the accepting states of the automaton, be it shortlex or geodesic. As in our case the tree $\Gamma$ spans the whole diagram $\mathcal{D}$, this gives a fair amount of freedom, which will be used later to prove strong connectivity of both automata.

For any given root $\alpha$ of $\Sigma$, let $\sigma_\alpha$ be the associated reflection. For a given set of simple roots $A = \{\alpha_{i_1}, \dots, \alpha_{i_k}\} \subset \mathbb{R}^n$, let $\mathrm{Fix}(i_1, \dots, i_k)$ be the set of all roots from $\Sigma$ that are fixed by $\sigma_\alpha$ for all $\alpha\in A$. 

\begin{lemma}[Fixed roots lemma]\label{lemma:stabiliser}
Let vertices $i$ and $j$ of $\mathcal{D}$ be adjacent in $\Gamma$. Then any element of $\mathrm{Fix}(i, j)$ belongs to the linear span of $\alpha_i + \alpha_j$ together with $\alpha_k$ for which $m_{ki}=2$ and $m_{kj}=2$ in the diagram $\mathcal{D}$, and vice versa. 
\end{lemma}
\begin{proof}
Let $v$ be a small root such that $\sigma_i(v)=v$ and $\sigma_j(v)=v$. Since $v$ is positive, we can write it as $v = \sum^n_{s=1} c_s \alpha_s$, with all $c_s \geq 0$ for $1 \leq s \leq n$ and at least one $c_s$ being non-zero. Since 
$\sigma_i$ and $\sigma_j$ fix $v$, then formula \eqref{eq:reflection-formula} gives
\begin{equation}
0 = (v | \alpha_i) = \sum^n_{s=1} c_s (\alpha_s | \alpha_i) = c_i (\alpha_i | \alpha_i) + c_j (\alpha_j | \alpha_i) + \sum^n_{s=1, s\neq i,j} c_s (\alpha_s | \alpha_i),
\end{equation}
\begin{equation}
0 = (v | \alpha_j) = \sum^n_{s=1} c_s (\alpha_s | \alpha_j) = c_j (\alpha_j | \alpha_j) + c_i (\alpha_i | \alpha_j) + \sum^n_{s=1, s\neq i,j} c_s (\alpha_s | \alpha_j).
\end{equation}
This yields, together with the fact that $-1 \leq (\alpha_s | \alpha_i) \leq 0$ and $-1 \leq (\alpha_s | \alpha_j) \leq 0$, for $s\neq i, j$, that
\begin{equation}
c_i - c_j = -\sum^n_{s=1, s\neq i,j} c_s (\alpha_s | \alpha_i) \geq 0,
\end{equation}
and, simultaneously,
\begin{equation}
c_i - c_j = \sum^n_{s=1, s\neq i,j} c_s (\alpha_s | \alpha_i) \leq 0.
\end{equation}
These two inequalities immediately imply that $c_i = c_j$. Then, we also see that $c_s = 0$ for all $s$ such that $\mathcal{D}$ has at least one of the edges connecting $s$ to $i$ or $j$.
\end{proof}

\begin{lemma}[Cycling lemma]\label{lemma:cycling}
Let some vertices $i$ and $j$ in the diagram $\mathcal{D}$ be connected by an edge in $\Gamma$. Then for any small root $v \in \Sigma = \Sigma(\mathcal{D})$, there exists a natural number $N \geq 1$ such that $(s_i s_j)^N(v) \notin \Sigma$, unless $v \in \mathrm{Fix}(i, j)$.
\end{lemma}
\begin{proof}
We shall prove that for any such $i$, $j$ and any positive root $v \notin \mathrm{Fix}(i, j)$, we have that
\begin{equation}
\lim_{k\to \infty} \| (s_i s_j)^k(v) \| \rightarrow \infty,
\end{equation}
in the $\ell_2$--norm. As $|\Sigma|<\infty$, this would imply the lemma.

Let $v_0 = v$, and let $R = s_i s_j$. By a straightforward computation,
\begin{equation}
R^k(v_0) = v_0 + (I + R + R^2 + \dots + R^{k-1})\, w,
\end{equation}
where
\begin{equation}
w = (-4(v|\alpha_i)-2(v|\alpha_j)) \alpha_i + (-2(v|\alpha_i)) \alpha_j = c_i \alpha_i + c_j \alpha_j.
\end{equation}
Then, by using the fact that $i$ and $j$ are connected by an edge in $\Gamma$, we compute
\begin{equation}
R(w) = R(c_i \alpha_i + c_j \alpha_j) = (3c_i - 2c_j) \alpha_i + (2c_i - c_j) \alpha_j.
\end{equation}
This means that in the subspace $W$ spanned by $\alpha_i$ and $\alpha_j$, the matrix of $R$ can be written as
\begin{equation}
R|_W = \left( \begin{array}{cc}
3& -2\\
2& -1
\end{array} \right),
\end{equation}
by using $\{\alpha_i, \alpha_j\}$ as a basis.
One can see that $R_W = T J_R T^{-1}$, where
\begin{equation}
J_R = \left( \begin{array}{cc}
1& 1\\
0& 1
\end{array} \right)
\end{equation}
is the Jordan normal form of $R|_W$, which has the following sum of powers:
\begin{equation}
R_k=\sum^{k-1}_{i=0} J_R^i = \left( \begin{array}{cc}
k& \frac{(k-1)k}2\\
0& k
\end{array} \right).
\end{equation}
As for any non-zero vector $u$ one has $\lim_{k\to \infty} \| R_k u \| = \infty$, we also get that
\begin{equation}
\|R^k(v_0) - v_0\| = \|\Big(  \sum^{k-1}_{i=0} R^i \Big) w\| = \| R_k(w) \| \rightarrow \infty,
\end{equation}
unless $w = 0$. In this case, by solving $c_i = c_j = 0$ about the inner products $(v|\alpha_i)$ and $(v|\alpha_j)$, we find that both inner products are equal to $0$, hence $v$ is stable under both reflections $\sigma_i$ and $\sigma_j$, which implies $v \in \mathrm{Fix}(i, j)$.
\end{proof}

The meaning of the Lemma above is that by repeated applications of  $s_i$ and $s_j$, which we informally call ``pedalling'', we can ``cycle away'' in the $\ell_2$--norm from any root $v$ and thus, in particular, we can escape any subset of small roots by applying Cycling lemma to its elements. We shall put this fact to essential use in one more lemma below.

In the following considerations we keep track of the coordinates in the canonical basis, so we introduce a notation $v[i]$ for the $i$--th coordinate $c_i$ of the vector $v$ written out as a sum $v = \sum^n_{s=1} c_s \alpha_s$ in the canonical basis of simple roots.

Then, for a finite set of positive roots $A \subset \mathbb{R}^n$, let us define its \textit{height} as
\begin{equation}
H(A) = \max_{v\in A}\, \{ i \, |\, v[i] \neq 0, \,\, v[j] = 0,\, \forall j > i \},
\end{equation}
and its \textit{width} as
\begin{equation}
W(A) = \mathrm{card}\, \{ v \in A\, |\, v[{h(A)}] \neq 0 \}.
\end{equation}

\begin{lemma}[Hydra's lemma]\label{lemma:hydra}
Let $D \neq \{ \emptyset \}, \star$ be a state of the automaton $\mathsf{ShortLex}$ or $\mathsf{Geo}$ for an $\infty$--spanned group $G$. Then there exists a word $w$ in the respective language such that $\widehat{\delta}(D, w) = \{ \alpha_1 \}$.
\end{lemma}
\begin{proof}

First we provide an argument in the case of the $\mathsf{ShortLex}$ automaton. Since by definition in each state $D\neq \{ \emptyset \}, \star$ there is a simple root, we choose some $\alpha_i \in D$. Also let $h = H(D)$ be the height of $D$ with $\mu \in D$ being some small root realising the height of $D$, i.e. $\mu[h] \neq 0$, while $\mu[k] = 0$, for all $h < k \leq n$. We also denote $F = \mathrm{Fix}(1, 2)$.

First, consider the case $h>2$. Our goal is to form a suitable word $w$ such that $w(\mu)\notin F$. Either $\mu \notin F$ right away, or one of the following cases holds.

\paragraph{I.}  \textit{There exists $k\in \{3,\ldots, h-1\}$ such that $\mu[k] \neq 0$.} Choose the minimal $k$ with this property, and let $(i_0,i_1,\ldots,i_p)$ be the path in the tree $\Gamma$ from the vertex $i_0=i$ towards the vertex $i_p=k$. Considering the words $w_l = s_{i_l} s_{i_{l-1}} \dots s_{i_1} s_{i_0}$, with $l\in \{1,\ldots,p-1 \}$, we may obtain that for some $l$ the vector $\mu' = \rho(w_l)(\mu)\notin F$. In this case, we move to the state $D' = \widehat{\delta}(D, w_l)$, which contains $\mu' = \rho(w_l)(\mu) \notin F$ and has $H(D') = H(D) = h$. Otherwise, we consider the word $w_p$, for which one has $\mu' = \rho(w_p)(\mu) \notin \mathrm{Fix}({i_{p-1}}, {i_p})$, hence we can apply Cycling lemma to $\mu'$. Thus, for some sufficiently large $N$ we have $\mu'' = (\sigma_{i_{p-1}} \sigma_{i_p})^N(\mu') \notin F$, and we move to the state $D' = \widehat{\delta}(D, w)$, with $w = (s_{i_{p-1}} s_{i_{p}})^N s_{i_{p-1}} \dots s_{i_2} s_{i_1}$ containing $\mu'' = \rho(w)(\mu) \notin F$, while $H(D') = H(D) = h$, since $w$ comprises only the reflections $s_l$ with $l<h$.

\paragraph{II.} \textit{For all $k\in \{3,\ldots,h-1 \}$ we have $\mu[k] = 0$.} Let $(i_0,i_1,\ldots,i_p)$ be the path in the tree $\Gamma$ from the vertex $i_0=i$ towards the vertex $i_p=h$. Again, moving up the tree $\Gamma$ by reading the word $w_l = s_{i_l} s_{i_{l-1}} \dots s_{i_1} s_{i_0}$, with $2 < l < p-1$, we either obtain that the vector $\mu' = \rho(w_l)(\mu)$ has a non-zero coordinate $k$ for some $2 < k < h$ and thus the state $D' = \widehat{\delta}(D, w_l)$ containing $\mu' = \rho(w_l)(\mu)$, satisfies Case I. Otherwise, we reach $l=p-2$, while in $\mu' = \rho(w_{p-2})(\mu)$ we have $\mu'[1] = \mu'[2] = c_1$ and $\mu'[h] = c_2 \neq 0$, with $\mu'[l] = 0$ for all other $2 < l < h$ and $h < l \leq n$.

If $\mu' \notin \mathrm{Fix}({i_{p-2}}, {i_{p-1}})$, we apply Cycling lemma as in Case I to remove the image of $\mu'$ from the state and thus decrease the width, and not increase the height.

If $\mu' \in \mathrm{Fix}({i_{p-2}}, {i_{p-1}})$, we either have $c_1=0$ or $(\alpha_{i_{p-1}}|\alpha_1) = 0$ and $(\alpha_{i_{p-1}}|\alpha_2) = 0$. In both cases, remembering that $(\alpha_{i_{p-1}}|\alpha_{i_{p}}) = -1$, we obtain that
\begin{equation}
\begin{aligned}
\mu'' &= \sigma_{i_{p-1}}(\mu') = \mu' - 2 (\alpha_{i_{p-1}}|\mu') \alpha_{i_{p-1}} = \\
 &= \mu' - 2 (c_1 (\alpha_{i_{p-1}}|\alpha_1) + c_1(\alpha_{i_{p-1}}|\alpha_2) + c_2(\alpha_{i_{p-1}}|\alpha_{i_p})) \alpha_{i_{p-1}} \\
 &= \mu' + 2 c_2 \alpha_{i_{p-1}},
\end{aligned}
\label{coord-computation}
\end{equation}
where $c_2\ne 0$. Then, $\mu''[i_{p-1}]=2c_2\ne 0 = \mu''[i_{p-2}]$, hence $\mu'' \notin \mathrm{Fix}({i_{p-2}}, {i_{p-1}})$. Then, again we can use the argument from Case I and apply Cycling lemma to $\mu''$. Indeed, taking $\mu''' = (\sigma_{i_{p-2}}\sigma_{i_{p-1}})^N(\mu'') \notin F$ for sufficiently big $N$ we obtain that $\mu'''\notin \Sigma$, so with a word $w = (s_{i_{p-2}}s_{i_{p-1}})^N s_{i_{p-2}} \dots s_{i_1} s_{i_0}$we move to the state $D' = \widehat{\delta}(D, w)$, which lacks $\mu''' = \rho(w)(\mu)$, while $H(D'') \le H(D) = h$.

By applying the above argument repeatedly, we arrive at an accept state $D^* = \widehat{\delta}(D, w)$ with a word $w$, possibly empty, in $\mathsf{ShortLex}$ or $\mathsf{Geo}$, so that $\lambda = \rho(w)(\mu)$ is contained in $D^*$, but not in $F$, for the above chosen $\mu \in D\cap F$. Also, we have $H(D^*) = H(D) = h$ and $W(D^*)\le W(D)$. This follows from the fact that all the roots $\lambda \in D^*$ realising the height of $D^*$ are images of height-realising roots $\mu \in D$. Indeed, no simple root $\alpha_k$ with $k\geq h$ has been added during the transition from $D$ to $D^*$, neither an image of such a root under a simple reflection $s_l$, with $l \geq h$. The word $w$ has only simple reflections $s_k$ with $k < h$, and thus we do not change any $k$--coordinates with $k\geq h$ for roots in $D$ and its subsequent images by applying any of the reflections in $w$.

Now, pick a height-realising root in $\lambda \in D^*$ and, since $\lambda \notin F$, apply Cycling lemma to $\lambda$ in order to arrive at a state $D_* = \widehat{\delta}(D^*, (s_1 s_2)^N)$, such that $H(D_*) \leq H(D) = h$, while $W(D_*) \leq W(D^*) - 1$. By applying this argument repeatedly, we can reduce the width of the subsequent states, and thus finally arrive at a state $\overline{D}$, such that $H(\overline{D}) \leq h - 1$. However, we have no control over the magnitude of $W(\overline{D})$, since many vectors of smaller height could have been added during all the above transitions\footnote{Thus, while chopping off hydra's bigger heads, we allow it to grow many more smaller ones, and nevertheless succeed in reducing it down to a single head remaining.}.

\medskip

We can apply the above argument, and finally bring the height of the state down to $h = 2$, hence all the roots in $D_*$ can be written as $c_1\alpha_1+c_2\alpha_2$. Due to Fixed root lemma, all roots in $D_*$ which are in $F = \mathrm{Fix}(1, 2)$ have $c_1 = c_2$. Since $\alpha_1+\alpha_2 = \sigma_1(\alpha_2) = \sigma_2(\alpha_1)$ is a small root, due to the dominance relation (c.f. the definitions \cite[p. 116]{BB} and \cite[Theorem~4.7.6]{BB}), this is the only option for the elements of $D_*\cap F$. Then, using Cycling lemma with powers of $(s_2 s_1)^N$ for sufficiently big $N\geq 1$ we can either reach $D_0 = \{\alpha_1\}$ or arrive to one of the states $D_1 = \{ \alpha_1, \alpha_1 + \alpha_2 \}$ or $D_2 = \{ \alpha_2, \alpha_1 + \alpha_2 \}$. Then, the states $D_1$ and $D_2$ form a two-cycle under the action of any word $w = (s_1 s_2)^N$, $N \geq 1$. Since $n \geq 3$, we use $s_3$ in order to transition instead from $D_1$ to $D_3 = \{\alpha_3, \beta_1 = \sigma_3(\alpha_1), \beta_2 = \sigma_3(\alpha_1 + \alpha_2) \}$. By Labelling lemma, vertices $2$ and $3$ are connected by an edge in $\Gamma$, and thus we can compute
\begin{equation}
\beta_1 = \alpha_1 - 2 (\alpha_1 | \alpha_3) \alpha_3 \notin F,
\end{equation}
by Fixed root lemma, since $\beta_1[1] = 1 \neq 0 = \beta_1[2]$, and
\begin{equation}
\beta_2 = \alpha_1 + \alpha_2 + (2 - 2(\alpha_1 | \alpha_3)) \alpha_3 \notin F,
\end{equation}
once again by Fixed roots lemma, since $\beta_2[3] \neq 0$  (recall that the inner product $(\alpha_1 | \alpha_3)$ is always non-positive), and the element $s_2 s_3$ has infinite order.

Now we can apply Cycling lemma to $D_3$ in order to move $\beta_1$ and $\beta_2$ away from the set $\Sigma$ of small roots, and finally arrive at the state $D_0 = \widehat{\delta}(D_3, (s_1 s_2)^N) = \{ \alpha_1, (\sigma_1 \sigma_2)^N(\beta_1), (\sigma_1 \sigma_2)^N(\beta_2) \} \cap \Sigma = \{ \alpha_1 \}$.

A similar argument applies to the case of $\mathsf{Geo}$ automaton, and it can be done by a simpler induction on $|D|$, the cardinality of $D$. Indeed, applying Hiking lemma never increases $|D|$, and applying Cycling lemma to the height-realising root reduces $|D|$.
\end{proof}

\begin{lemma}[GCD lemma]\label{lemma:GCD}
The greatest common divisor of the lengths of all cycles in the $\mathsf{ShortLex}$, resp. $\mathsf{Geo}$, automaton for an $\infty$--spanned Coxeter group equals $1$.
\end{lemma}
\begin{proof}
First of all, let us notice that there is a cycle of length $2$ in each:
\begin{equation}
\{ \alpha_1 \} \rightarrow \delta(\{ \alpha_1 \}, s_2) = \{ \alpha_2 \} \rightarrow \delta(\{ \alpha_2 \}, s_1) = \{ \alpha_1 \}.
\end{equation}

Then, let us consider the following sequence of transitions in $\mathsf{ShortLex}$. Let $D_0 = \{ \alpha_1 \}$, and let $m_{13} \neq \infty$. Then $D_1 = \delta(D_0, s_3) = \{ \alpha_3, \mu = \alpha_1 + c \alpha_3  \}$, where $c = - 2 \cos \frac{\pi}{m_{13}} \leq 0$. Here, $\mu \notin \mathrm{Fix}({1, 2})$ by Fixed roots Lemma. Thus, there exists a natural number $N \geq 1$ such that $(\sigma_1 \sigma_2)^N(\mu) \notin \Sigma$, and $D_{2N + 1} = \widehat{\delta}(D_1, (s_1s_2)^N) = \{ \alpha_1 \} = D_0$. This means that we obtain a cycle of odd length.  If $m_{13} = \infty$, then $\mu \notin \Sigma$, and we readily obtain a cycle of length $3$ by putting $N = 1$.

A similar argument applies to the case of $\mathsf{Geo}$ automaton.
\end{proof}

\section{Proofs of main theorems}\label{proofs}

In this section we use the auxiliary lemmas obtained above in order to prove the main theorems of the paper. Namely, we show that the following statement hold for a Coxeter group $G$ that is $\infty$--spanned:
\begin{itemize}
\item the word growth rate $\omega(G) \geq (1+\sqrt{5})/2$ and the geodesic growth rate 
$\gamma(G) \geq \gamma_0$ (where $\gamma_0$ is the real root of $x^3 - 2x - 2$) are Perron numbers (Theorem~\ref{thm:Perron});
\item unless $G$ is a free product of more than $2$ copies of $C_2$, we have $\gamma(G) > \omega(G)$ (Theorem~\ref{thm:geodesic}).
\end{itemize}

\paragraph{Proof of Theorem~\ref{thm:Perron}.} Below, we show that the word growth rate $\omega(G)$ of an $\infty$--spanned Coxeter group $G$ (with respect to its standard generating set) is a Perron number. A fairly analogous argument shows that the geodesic growth rate $\gamma(G)$ of $G$ is also a Perron number.

First, we show that any state $D = \widehat{\delta}(\{ \emptyset \}, w)$, for a shortlex word $w$, can be reached from the state $\{ \alpha_1 \}$. Observe, that $\delta(\{ \alpha_1 \}, s_k) = \{ \alpha_k \} \cup \{ s_k(\alpha_1), s_k(\alpha_l), l < k \} \cap \Sigma = \{ \alpha_k, s_k(\alpha_l), l < k \} \cap \Sigma = \delta(\{\emptyset \}, s_k)$, for any $k > 1$. Thus, $\widehat{\delta}(\{\emptyset\}, w) = \widehat{\delta}(\{\alpha_1\}, w)$, if $w$ does not start with $s_1$, and $\widehat{\delta}(\{\emptyset\}, w) = \widehat{\delta}(\{\alpha_1\}, w')$, if $w = s_1 w'$.

Then, Hydra's lemma guarantees that we can descend in $\mathsf{ShortLex}$ from any state $D \neq \star$ to $\{ \alpha_1 \}$. Together with the above fact, we have that $\mathsf{Geo}\setminus \{ \emptyset \}$ is strongly connected, and then the transfer matrix $M = M(\mathsf{Geo}\setminus \{ \emptyset \})$ is irreducible.

By GCD lemma, $M$ is also aperiodic, and thus primitive. Then the spectral radius of $M$ is a Perron number \cite[Theorem 4.5.11]{LM}.  Since the latter equals the growth rate of the shortlex language for $G$  by \cite[Proposition 4.2.1]{LM}, we obtain that $\omega(G)$ is a Perron number.

In order to prove the lower bounds for $\omega(G)$ and $\gamma(G)$ we first note that for any $\infty$--spanned Coxeter group $G$ there exist three generators $s_i,s_j,s_k\in S$ such that $m_{ij}=m_{ik}=\infty$ and $m_{kj}=n$, which can be either some finite label $n\geq 2$ or a label $n=\infty$. Then the parabolic subgroup $\langle s_i,s_j,s_k \rangle$ is the free product $C_2 * D_n$, where $D_n$ is either the dihedral group of order $2n$, if $n$ is finite, or $D_\infty$, if $n = \infty$.

For the free product $A*B$ of two groups $A$ and $B$ generated by sets $U$ and $V$, respectively, we have the following formula that relates the growth series $\omega_{(A*B,U\cup V)}(z)$ to the series $\omega_{(A,U)}(z)$ and $\omega_{(B,V)}(z)$, cf. \cite[p.~156]{H}:
\begin{equation}\label{free_product-formula}
  \frac{1}{\omega_{(A*B, U\cup V)}(z)}=\frac{1}{\omega_{(A, U)}(z)}+\frac{1}{\omega_{(B, V)}(z)}-1.
\end{equation}
Since for $U$ and $V$ being the standard generating sets of $C_2$ and $D_2$, respectively, we have $\omega_{(C_2, U)}(z) = 1 + z$ and $\omega_{(D_2, V)}(z) = 1 + 2z + z$, formula \eqref{free_product-formula} implies $\omega_{(C_2 * D_2, U\cup V)}(z)=(x+1)^2/(1-x-x^2)$. The smallest real pole of this function is $1/\varphi$, where $\varphi=(1+\sqrt5)/2$, hence $\varphi$ is the growth rate of
$C_2 * D_2$ with respect to the standard generating set consisting of three involutions. Note that the shortlex language for $C_2 * D_2$ is a sublanguage of the shortlex language for $C_2 * D_n$, hence for the word growth rates we have $\omega(C_2 * D_2)\leq \omega(C_2 * D_n)$ for any $n\geq 2$. Since $\langle s_i,s_j,s_k \rangle = C_2 * D_n$ is a parabolic subgroup of $G$, we obtain the desired inequality of $\omega(G)\geq \varphi$.

A formula completely similar to \eqref{free_product-formula} was proved in \cite[p.~753]{LMW} for the geodesic growth series of a free product: one just needs to substitute all instances of $\omega$ by $\gamma$. Plugging $\gamma_{(C_2, U)}(z)=1+z$ and $\gamma_{(D_2, V)}(z)=1+2z+2z^2$ into this formula, we obtain that $\gamma_{(C_2 * D_2, U\cup V)}(z)=(x+1)(1+2x+2x^2)/(1-2x^2-2x^3)$. One can easily check that $1/\gamma_0$ has the smallest absolute value among the poles of $\gamma_{C_2 * D_2}(z)$, hence $\gamma(C_2 * D_2) = \gamma_0$. The natural inclusion of the geodesic languages shows that $\gamma(G)\geq \gamma_0$, analogous to the argument above. \qed

\begin{remark}
The conditions under which Theorem~\ref{thm:Perron} was proved may seem somewhat tight, as for a Coxeter group of rank $n$ at least $n-1$ edges have to be labelled $\infty$. One could try to loosen this restriction and substitute some of these $\infty$ labels by a large natural number and expect an analogous statement to hold. However, it seems hardly possible to adapt our approach in this case, as the auxiliary lemmas would then need to be stated in a quantitative way, while we have no control whatsoever neither on the cardinality of the set $\Sigma$, nor on the norms of the small roots in it. This seems indeed a major obstacle, cf. \cite{Casselman2} for more details.
\end{remark}

\paragraph{Proof of Theorem~\ref{thm:geodesic}.} Next, we aim at proving that $\gamma(G) > \omega(G)$, unless $G$ is a free product of several copies of $C_2$, in which case $\gamma(G) = \omega(G)$. For convenience, let $A$ denote the automaton $\mathsf{ShortLex}$ and $B$ denote the automaton $\mathsf{Geo}$ for $G$. Let $L(F)$ be the language accepted by a given finite automaton $F$, and let $\lambda(F)$ be the exponential growth rate of $L$.

We shall construct a new automaton $A'$, by modifying $A$, such that
\[L(A) \subsetneq L(A') \subseteq L(B)\]
and, moreover, $\omega(G) = \lambda(A) < \lambda(A') \leq \lambda(B) = \gamma(G)$.

\begin{figure}[ht]
\centering
\includegraphics[width=0.5\textwidth]{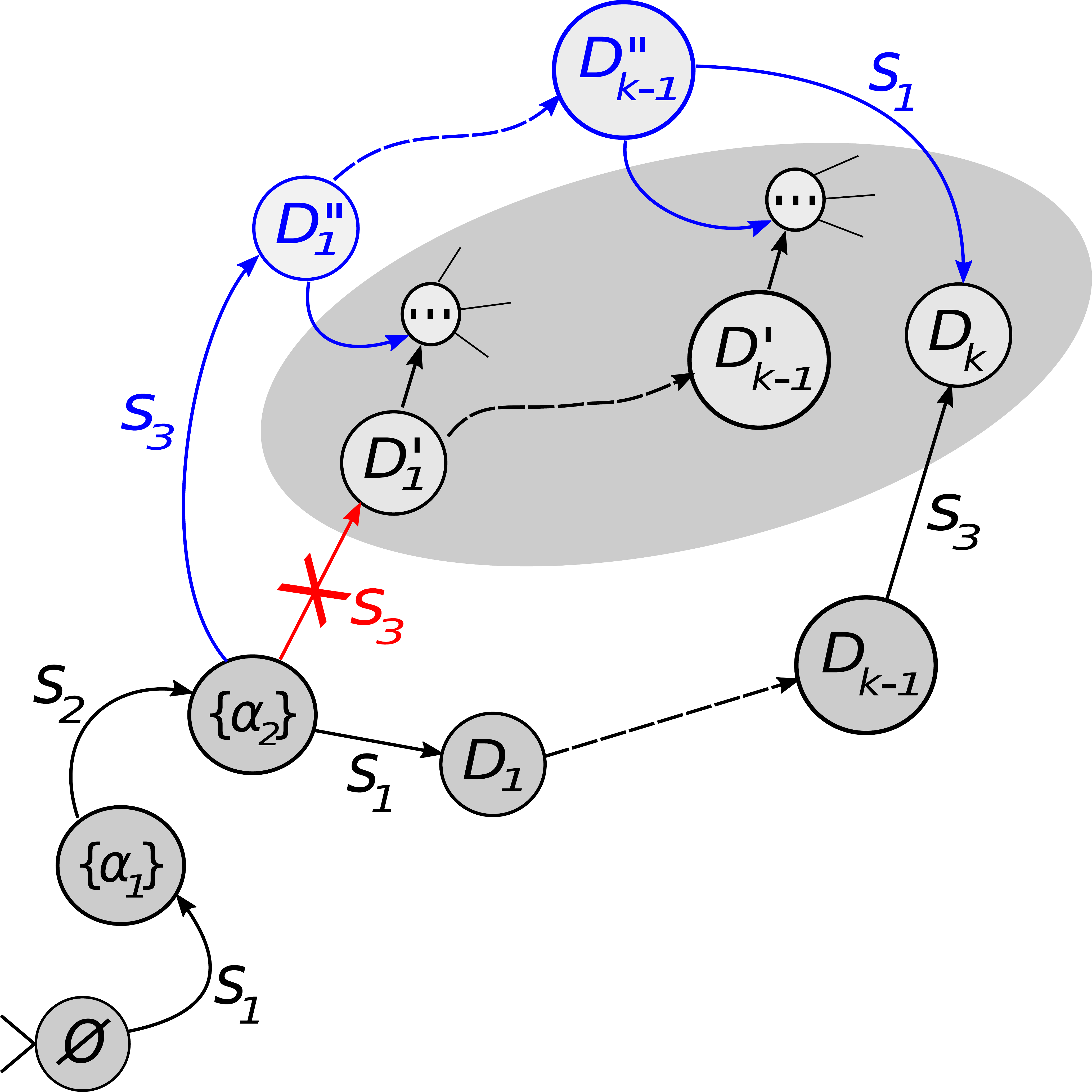}
\caption{The modified automaton $A'$: transition $\{\alpha_2\} \rightarrow D'_1$ is removed and a path $p$ comprising new states $D''_i$ is added.}\label{fig:automaton}
\end{figure}

Since $G$ is not a free product, we  may assume that the edge $1 \rightarrow 3$ has label $m \geq 2$, and $m \neq \infty$. Consider two cases depending on the parity of $m$. If $m$ is even, then let $w = s_1 s_2 (s_1 s_3)^{m/2}$, $w' = s_1 s_2(s_3 s_1)^{m/2-1}s_3$ and $w'' = s_1 s_2(s_3 s_1)^{m/2}$. If $m$ is odd, then $w = s_1 s_2(s_1 s_3)^{(m-1)/2} s_1$, $w' = s_1 s_2(s_3 s_1)^{(m-1)/2}$ and $w'' = s_1 s_2(s_3 s_1)^{(m-1)/2} s_3$. We shall use the straightforward equality $w=w''$ which holds for $w$ and $w''$ considered as group elements. One can also verify that in both cases $w, w' \in L(A)$ and $w'' \in L(B) \setminus L(A)$.

Let the word $w$ correspond to the directed path $\{\emptyset\} \rightarrow \{ \alpha_1 \} \rightarrow \{ \alpha_2 \} \rightarrow D_1 \rightarrow \dots \rightarrow D_k$, and the word $w'$ correspond to the directed path $\{\emptyset\} \rightarrow \{ \alpha_1 \} \rightarrow \{ \alpha_2 \} \rightarrow D'_1 \rightarrow \dots \rightarrow D'_{k-1}$ in $A$.
Then, let the graph $A'$ be obtained from $A$ in the following way, which is schematically illustrated in Figure~\ref{fig:automaton}:
\begin{itemize}
\item[1)] Add a number of states $D''_1$, $D''_2$, $\dots$, $D''_k$ to $A$, and create a directed path $ p =\{ \emptyset \} \rightarrow \{ \alpha_1 \} \rightarrow \{ \alpha_2 \} \rightarrow D''_1 \rightarrow \dots \rightarrow D''_{k-1} \to D_k$ in $A$ labelled with the sequence of letters in $w''$. Let $\varepsilon$ be the last edge of $p$.
\item[2)] Remove the transition $\{ \alpha_2 \} \rightarrow D_1'$ labelled by $s_3$, and for all $1 \leq i \leq k-1$ add $n-1$ transitions $D''_i \rightarrow \delta(D_i', s_j)$, where $s_j$ runs over all labels except one that is already used for the transition $D''_i\to D''_{i+1}$.
\item[3)] Let $A'$ be the subgraph in the automaton above spanned by the start state $\{ \emptyset \}$ together with the strongly connected component of
$\{ \alpha_1 \}$, which (by the fact that $A\setminus \{ \emptyset \}$ is strongly connected) is equivalent to removing all inaccessible states.
\end{itemize}

Let us define yet another automaton $A''$, which is obtained from $A'$ be removing the only transition $\varepsilon$. It follows from points (2)--(3) in the definition of $A'$ above that all the states $D''_i$, $k-1 \geq i \geq 1$ belong to the strongly connected component of $\{ \alpha_1 \}$, and thus we do not create any inaccessible states in $A''$ by removing $\varepsilon$ from $A'$.

Observe, that we have $L(A'') = L(A)$, since each word $u$ accepted by $A''$ can be split into two types of subwords: subwords read while traversing a sub-path of $p$, and subwords read while traversing paths that consist of the states of the original automaton $A$. However, each subword $v$ of $u$ obtained by traversing a subpath of $p$ can be obtained by traversing the states of $A$, since $v$ is a subword of $w''$, but $v\neq w$. Thus, $L(A'') \subset L(A)$. The inclusion $L(A) \subset L(A'')$ follows by construction.

On the other hand, $L(A) \subsetneq L(A') \subseteq L(B)$, since $w'' \in L(A')$, while $w'' \notin L(A)$.

From the above description, we obtain that $A'\setminus \{\emptyset\}$ and $A''\setminus \{\emptyset\}$ are both strongly connected. Then the transition matrices $M' = M(A'\setminus \{ \emptyset \})$ and $M'' = M(A'' \setminus \{ \emptyset \})$ are both irreducible.
Moreover, $M'$ and $M''$ have same size and $M' \neq M''$ dominates $M''$, since $A'$ and $A''$ have an equal number of states, while $A''$ has fewer transitions than $A'$. Then \cite[Corollary A.9]{B} implies that  $\lambda(A) = \lambda(A'') < \lambda(A') \leq \lambda(B)$, and thus $\omega(G) = \lambda(A) < \lambda(B) = \gamma(G)$.

Since Perron--Frobenius eigenvalues are simple, the quantities $w_n$ and $g_n$ asymptotically satisfy $w_n \sim \nu_1\cdot \omega(G)^n$ and $g_n \sim \nu_2\cdot \gamma(G)^n$, as $n\rightarrow \infty$, for some constants $\nu_1, \nu_2 > 0$. Then the remaining claim of the theorem follows. \qed

\section{Geometric applications}\label{sec:geom}

In this section we bring up some applications of our result to reflection groups that act discretely by isometries on hyperbolic space $\mathbb{H}^n$. A convex polytope $P \subset \mathbb{H}^n$, $n\geq 2$, is the intersection of finitely many geodesic half-spaces, i.e. half-spaces of $\mathbb{H}^n$ bounded by hyperplanes. A polytope $P \subset \mathbb{H}^n$ is called Coxeter if all the dihedral angles at which its facets intersect are of the form $\frac{\pi}{m}$, for integer $m\geq 2$.

The geometric Coxeter diagram $\mathcal{D}$ of $P$ is obtained by indexing its facets with a finite set of consecutive integers $F = $ $\{1$, $2$, $\dots \}$, and forming a labelled graph on the set of vertices $F$ as follows. If facets $i$ and $j$ intersect at an angle $\frac{\pi}{m_{ij}}$, then the vertices $i$ and $j$ are connected by an edge labelled $m_{ij}$, if $m_{ij} \geq 4$; by a single unlabelled edge, if $m_{ij}=3$; or no edge is present, if $m_{ij}=2$. If facets $i$ and $j$ are tangent at a point on the ideal boundary $\partial \mathbb{H}^n$, then $i$ and $j$ are connected by a bold edge. If the hyperplanes of $i$ and $j$ admit a common perpendicular, i.e. do not intersect in $\overline{\mathbb{H}^n} = \mathbb{H}^n \cup \partial \mathbb{H}^n$, then $i$ and $j$ are connected by a dashed edge.

It is known that a Coxeter polytope  $P\subset \mathbb{H}^n$ gives rise to a discrete reflection group generated by reflections in the hyperplanes of the facets of $P$. The group $G = G(P)$ generated by $P$ is a Coxeter group with standard generating set $S$ given by facet reflections. Then the word growth rate $\alpha(G)$ and geodesic growth rate $\gamma(G)$ with respect to $S$ are be defined as usual. The diagram $\mathcal{D}$ of $G$ as a Coxeter group can be obtained from the diagram of $P$ by converting all bold and dashed edges, if any, into $\infty$--edges.

Usually, the polytope $P$ is assumed to be compact or finite-volume, i.e. non-compact and such that its intersection with the ideal boundary $\partial \mathbb{H}^n$ consists only of a finite number of vertices. This condition can be relaxed in our case, since it does not particularly influence any of the statements below.

Since the facets of a Coxeter polytope $P\subset \mathbb{H}^n$ intersect if and only if their respective hyperplanes do \cite{Andreev}, then the number and incidence of $\infty$--edges in the diagram of $G = G(P)$ is determined only by the combinatorics of $P$.

The following two facts show that many Coxeter group acting on $\mathbb{H}^n$, $n\geq 2$, discretely by isometries have Perron numbers as their word and geodesic growth rates.

\begin{theorem}\label{thm:Coxeter1}
Let $P \subset \mathbb{H}^n$, $n\geq 2$, be a finite-volume Coxeter polytope, and $G$ its associated reflection group. If the bold and dashed edges in the diagram of $P$ form a connected subgraph, then $\alpha(G)$ and $\gamma(G)$ are Perron numbers.
\end{theorem}

The connectivity condition above can be checked for the diagram of $P$ relatively easily either by hand, for small diagrams, or by using a computer program, for larger ones. It is also clear that Theorem \ref{thm:Coxeter1} is just a restatement of Theorem \ref{thm:Perron}.

An additional fact holds as we compare the word and geodesic growth rates of Coxeter groups of the above kind.

\begin{theorem}\label{thm:Coxeter2}
Let $P \subset \mathbb{H}^n$, $n\geq 3$, be a finite-volume Coxeter polytope, and $G$ its associated reflection group. If the bold and dashed edges in the diagram of $P$ form a connected subgraph, then $\alpha(G) < \gamma(G)$.
\end{theorem}
\begin{proof}
Let us notice that, unless $n=2$, it is impossible for a Coxeter polytope $P$ to have finite volume given that $\Gamma$ is a complete graph (in dimension $2$ we have an ideal triangle and its reflection group is isomorphic to the free product $C_2 * C_2 * C_2$). Indeed, let us consider an edge stabiliser of $P$. Since $P$ has finite volume, $P$ is \textit{simple at edges}, meaning that each edge is an intersection of $n-1$ facets. Then the edge stabiliser has a Coxeter diagram that is a subdiagram spanned by $n-1\geq 2$ vertices in the complete graph on $f$ vertices. Thus, it is itself a complete graph that has $\infty$--labels on its edges. This cannot be a diagram of a finite Coxeter group, hence Vinberg's criterion \cite[Theorem~4.1]{Vinberg} is not satisfied, and $P$ cannot have finite volume. Thus, $G$ cannot be a free product of finitely many copies of $C_2$, and the conditions of Theorem~\ref{thm:geodesic} are satisfied.
\end{proof}

As follows from the results by Floyd \cite{Floyd} and Parry \cite{Parry}, if $P$ is a finite-area polygon in the hyperbolic plane $\mathbb{H}^2$, the word growth rate $\alpha(G)$ of its reflection group $G$ is a Perron number. More precisely, $\alpha(G)$ is a Salem number if $P$ is compact, and a Pisot number if $P$ has at least one ideal vertex. A similar result holds for the geodesic growth rate $\gamma(G)$.

\begin{theorem}\label{thm:Coxeter3}
Let $P \subset \mathbb{H}^2$ be a finite-volume Coxeter polygon, and $G$ its associated reflection group. Then $\gamma(G)$ is also a Perron number whenever $P$ has more than $4$ vertices, or when $P$ is a quadrilateral with at least one ideal vertex, or a triangle with at least two ideal vertices. In all the above mentioned cases, $\gamma(G) > \alpha(G)$ unless $P$ is ideal.
\end{theorem}
\begin{proof}
The proof proceeds case-by-case based on the number of sides of $P$.

\textit{$P$ is a triangle.} If $P$ has two or three ideal vertices, then the subgraph of bold edges in the diagram of $D$ is connected. This subgraph is complete if and only if $P$ is an ideal triangle.

\textit{$P$ is a quadrilateral.} If $P$ has at least one ideal vertex, then the subgraph of bold and dashed edges in the diagram of $G$ is connected. This subgraph is complete if and only if $P$ is an ideal quadrilateral.

\textit{$P$ has $n\geq 5$ sides.} In this case, each vertex in the diagram of $G$ is connected by dashed edges to $n-3$ other vertices. It can be also connected by bold edges to one or two more vertices, depending on $P$ having vertices on the ideal boundary $\partial \mathbb{H}^2$. Provided the vertex degrees, it is clear that the subgraph of bold and dashed edges in $D$ is connected. This subgraph is complete if and only if each vertex in the diagram of $D$ is connected to $n-3$ vertices by dashed edges, and to two more vertices by bold edges. In this case, $P$ is an ideal $n$--gon.

Having described the cases above, the theorem follows from Theorems \ref{thm:Coxeter1} -- \ref{thm:Coxeter2}.
\end{proof}

Another series of examples where Theorems \ref{thm:Coxeter1} -- \ref{thm:Coxeter2} apply arises in $\mathbb{H}^3$: these are the right-angled L\"obell polyhedra originally described in \cite{Loebell} and their analogues with the same combinatorics but various Coxeter angles \cite{BMV, V}. The latter polyhedra can be obtained from the L\"obell ones by using ``edge contraction'', c.f. \cite[Propositions 1 -- 2]{K}. A few examples also come from hyperbolic Coxeter groups associated with quadratic integers \cite{JohnsonWeiss}.

The word growth rates of their associated reflection groups are Perron numbers by \cite{Yu1, Yu2}, and their geodesic growth rates are Perron numbers by Theorem~\ref{thm:Coxeter1}. Indeed, any Coxeter polyhedron $P$ polyhedron combinatorially isomorphic to a L\"obell polyhedron $L_n$ has the following property: each of its faces has at most $n$ neighbours, while $L_n$ has $2n+2$ faces in total. This implies that there are enough common perpendiculars in between its faces to keep the subgraph of dashed edges in the Coxeter diagram of $P$ connected. Also, Theorem~\ref{thm:Coxeter2} implies that the geodesic growth rates always strictly dominate the respective word growth rates.

\begin{figure}[ht]
\centering
\includegraphics[width=0.9\textwidth]{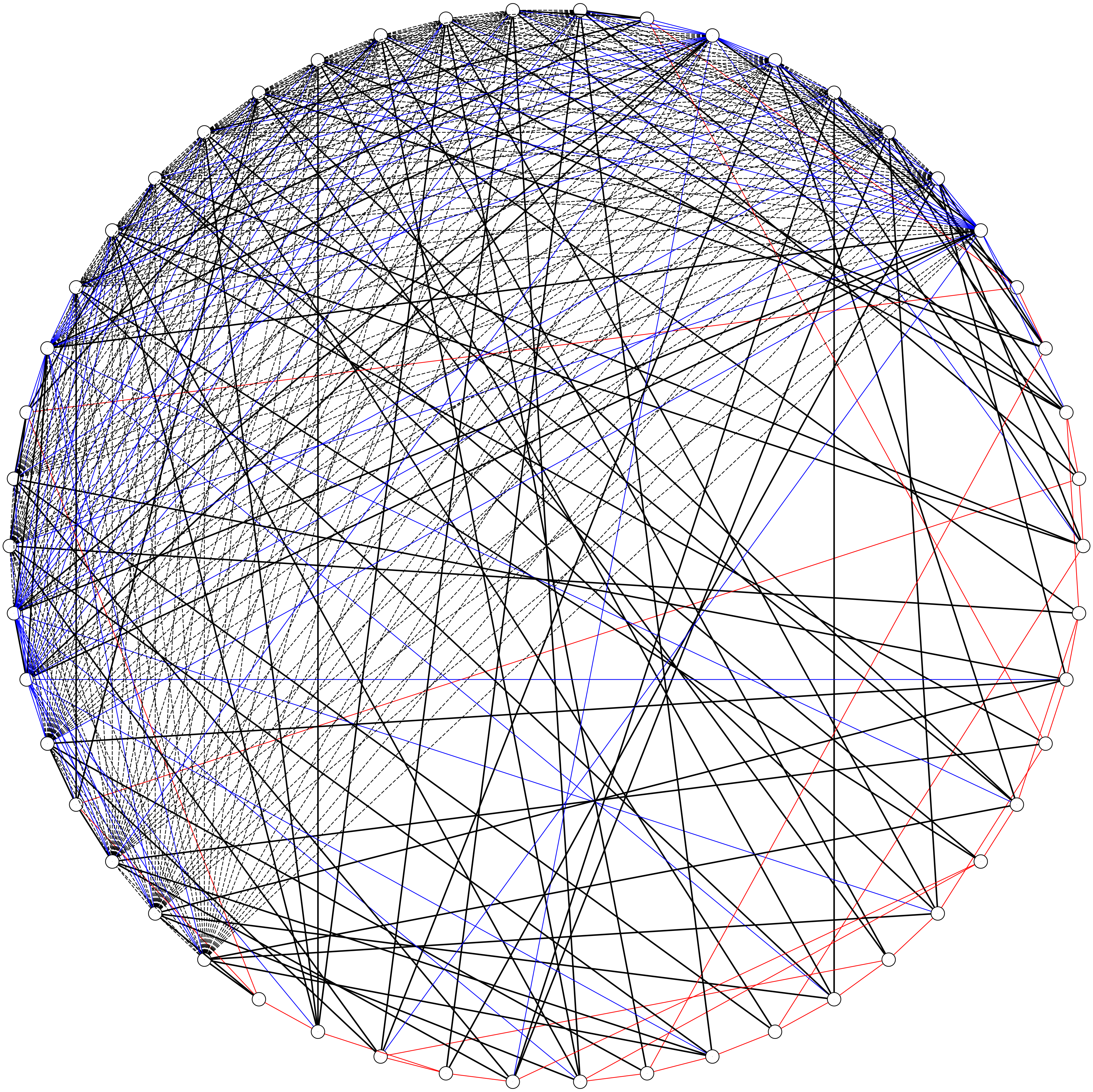}
\caption{A finite-volume non-compact Coxeter polytope in $\mathbb{H}^{19}$ }\label{fig:polytope-19}
\end{figure}

In Figure~\ref{fig:polytope-19}, we present a complete Coxeter diagram of the hyperbolic finite-volume polytope $P$ in $\mathbb{H}^{19}$ discovered by Kaplinskaya and Vinberg in \cite{KapVin}. The reflection group $G$ associated with $P$ corresponds to a finite index subgroup in the group of integral Lorentzian matrices preserving the standard hyperboloid $H = \{(x_0, x_1, \ldots, x_{19}) \in \mathbb{R}^{20} \,\, | \,\, -x^2_0 + x^2_1 + \ldots + x^2_{19} = -1, \,\, x_0 > 0\}$. The latter group is isomorphic to $G \rtimes S_5$, where $S_5$ is the symmetric group on $5$ elements. The diagram in Figure~\ref{fig:polytope-19} was obtained by using AlVin \cite{Guglielmetti-1, Guglielmetti-2} software implementation of Vinberg's algorithm \cite{Vinberg-algorithm}. The picture does not exhibit the $S_5$ symmetry but rather renders the edges as sparsely placed as possible in order to let the connectivity properties of the graph be observed.

The dashed edges correspond to common perpendiculars between the facets, and bold edges correspond to facets tangent at the ideal boundary $\partial \mathbb{H}^{19}$. The blue edges have label $4$, and the red ones have label $3$ (because of the size of the diagram, this colour notation seems to us visually more comprehensible).

Checking that the subgraph of bold and dashed edges in the diagram of $P$ is connected can be routinely done by hand or by using a simple computer program. Then Theorems \ref{thm:Coxeter1} -- \ref{thm:Coxeter2} apply.  We would like to stress the fact that checking whether the word and geodesic growth rates of $G$ satisfy the conclusions of Theorems \ref{thm:Coxeter1} -- \ref{thm:Coxeter2} by direct computation would be rather tedious, especially for the geodesic growth rate.

\addcontentsline{toc}{section}{References}

\footnotesize{
Authors' affiliations:\\

Alexander Kolpakov\\
Institut de math\'ematiques, Rue Emile-Argand 11, 2000 Neuch\^atel, Switzerland\\
Laboratory of combinatorial and geometric structures, Moscow Institute of Physics \\ and Technology, Dolgoprudny, Russia\\
kolpakov (dot) alexander (at) gmail (dot) com\\

Alexey Talambutsa\\
Steklov Mathematical Institute of RAS, 8 Gubkina St., 119991 Moscow, Russia\\
HSE University, 11 Pokrovsky Blvd., 109028 Moscow, Russia\\
alexey (dot) talambutsa (at) gmail (dot) com

}

\end{document}